\documentclass[12pt,twoside]{article}
\usepackage{mathrsfs}
\usepackage{amsthm}
\usepackage{amsmath}
\usepackage{amsbsy}
\usepackage{relsize}
\usepackage[lmargin=3.5cm,rmargin=3.5cm,tmargin=4cm,bmargin=4cm]{geometry}

\date{}

\begin{document}

\centerline{\bf }

\centerline{\bf }

\centerline{\bf }

\centerline{}

\centerline{}

\centerline{\Large{\bf On the fractional mixed fractional Brownian motion }}

\centerline{}

\centerline{\Large{\bf  Time Changed by Inverse $\alpha$-Stable Subordinator}}

\centerline{}

\centerline{\bf { Ezzedine Mliki}}

\centerline{}

\centerline{Department of Mathematics}

\centerline{ College of Science, Imam Abdulrahman Bin Faisal University}

\centerline{Basic and Applied Scientific Research Center}

\centerline{ P.O. Box 1982, Dammam, 31441, Saudi Arabia}

\centerline{ermliki@iau.edu.sa}

\newtheorem{Theorem}{\quad Theorem}[section]

\newtheorem{Definition}[Theorem]{\quad Definition}

\newtheorem{Corollary}[Theorem]{\quad Corollary}

\newtheorem{Lemma}[Theorem]{\quad Lemma}

\newtheorem{Example}[Theorem]{\quad Example}

\newtheorem{Notation}[Theorem]{\quad Notation}

\newtheorem{Remark}[Theorem]{\quad Remark}

{\footnotesize }

\centerline{}

\begin{abstract}
A time-changed fractional mixed fractional Brownian motion by inverse $\alpha$-stable subordinator with index $\alpha\in(0, 1)$   is an iterated process $L^{H_1H_2 }_{T^{\alpha}}(a,b)$ constructed as the superposition of fractional mixed fractional Brownian motion $N^{H_1H_2 }(a,b)$ and  an independent inverse $\alpha$-stable subordinator $T^{\alpha}$.  In this paper we prove that the process  $L^{H_1H_2 }_{T^{\alpha}}(a,b)$  is of long range dependence property under a smooth condition on the Hirsh index $H_1$ and $ H_2$. We deduce that the fractional mixed fractional Brownian motion  has long range dependence for  every $H_1 < H_2.$
\end{abstract}

{\bf Mathematics Subject Classification:} {60G20; 60G18; 60G15; 60G10 } \\

{\bf Keywords:} {Fractional  Mixed Fractional Brownian Motion; Long-range Dependence;  Inverse $\alpha$-stable subordinator}

\section{Introduction}
  A mixed fractional Brownian motion (mfBm for short) of parameters $a,  b $ and $  H $ is the process  $M^H(a,b)=\{M^{H}_{t}(a,b), \; t\geq 0\}$, defined on the probability space  $(\Omega, \mathcal{F}, P)$ by
	\begin{eqnarray*}
		M_t^H(a,b)=aB_t+bB_t^H, \quad t\geq0
	\end{eqnarray*}
	where  $B=\left\lbrace B_t, t\geq0\right\rbrace $ is a Brownian motion,   $B^H=\left\lbrace B_t^H, t\geq0\right\rbrace $ is an independent fractional Brownian motion of Hurst exponent $H\in(0,1)$ and  $a$, $b$ two real constants such that $(a, b)\neq (0, 0)$. The mfBm Was introduced by Cheridito \cite{Cher}, 
 with stationary increments exhibit a long-range dependence for $H>\frac{1}{2}.$ The mixed fractional Brownian motion has been discussed in \cite{Cher} to present a stochastic model of the discounted stock price in some arbitrage-free and complete financial markets. This model is the process
\begin{eqnarray*}
X_{t}=X_{0}\exp\{\mu t+\sigma (aB_t+bB_t^H)\},
\end{eqnarray*}
where $\mu$ is the rate of the return and $\sigma$ is the volatility. We refer also to \cite{ElNo, LIN, Thale} for further information and applications on the mfBm.

The  time-changed mixed fractional Brownian motion by  inverse $\alpha$-stable subordinator with index $\alpha\in(0, 1)$ is defined as  below
\begin{eqnarray*}
L^{H}_{T^{\alpha}}(a,b)=\{M^{H}_{T^{\alpha}_{t}}(a,b), \; t\geq 0\},
\end{eqnarray*}
where the parent process $N^{H}(a,b)$ is a mfBm with parameters $a, b,$ $H\in (0, \, 1)$ and  $T^{\alpha}=\{ T^{\alpha}_{t},\, t\geq 0\}$  is an  inverse $\alpha$-stable subordinator assumed to be independent of both Brownian and  fractional  Brownian motion. If $H=\frac{1}{2}$, the process $L^{\frac{1}{2}}_{T^{\alpha}}(0,1)$ is called subordinated Brownian motion, it was investigated in \cite{ GSPE, Mag, MMEM, Nan}. When $a=0,$ $b=1$ then $L^{H}_{T^{\alpha}}(0,1)$ it is the process  considered in   \cite{KGW, KWPS} called subordinated fractional Brownian motion.

Time-changed process is constructed by taking superposition of tow independent stochastic systems. The evolution of time in external process is replaced by a non-decreasing stochastic process, called subordinator.  The resulting time-changed process very often retain important properties of the external process, however certain characteristics might change. This idea of subordination was introduced by Bochner \cite{Bochner} and  was  explored in many papers (see  \cite{HmMl, MeHmMl, KGW, MeMl}).

The  time-changed mixed fractional Brownian motion has been discussed in \cite{GZH} to present a stochastic Black-Scholes model, whose price of the underlying stock is the process
\begin{eqnarray*}
S_{t}=S_{0}\exp\{\mu T^{\alpha}_{t}+\sigma( aB_{T^{\alpha}_{t}}+bB^H_{T^{\alpha}_{t}})\},
\end{eqnarray*}
where $\mu$ is the rate of the return, $\sigma$ is the volatility and $T^{\alpha}$ is the $\alpha$-inverse stable subordinator. Also
the time-changed processes have found many interesting applications, for example in finance  \cite{GZH, HJY, Omer, FSH, ZHH}.

 C. Elnouty  \cite{ElNo} propose a generalisation of the mfBm called fractional mixed fractional Brownian motion (fmfBm) of parameters $a,  b $ and Hirsh index $ H=( H_1, H_2).$ A fmfBm is a process  $N^H(a,b)=\{N^{H}_{t}(a,b), \; t\geq 0\}$, defined on the probability space  $(\Omega, \mathcal{F}, P)$ by
	\begin{eqnarray*}
		N_t^{H_1H_2}(a,b)=aB^{H_1}_t+bB^{H_2}_t,\quad t\geq 0,
	\end{eqnarray*}
	 where $B=\left\lbrace B_t, t\geq0\right\rbrace $ is a Brownian motion and  $B{^{H_{i}}}=\left\lbrace B_t^{H_{i}}, t\geq0\right\rbrace $ are independent fractional Brownian motion of Hurst exponent $H_{i}\in(0,1)$ for $i=1,2.$ Also the fmfBm was study by Miao, Y et al. \cite{Mia}.

The time-changed fractional mixed fractional Brownian motion is defined as 	 \begin{eqnarray*}L^{H}_{\beta}(a,b)=\{L^{H_1H_2}_{\beta_{t}}(a,b), \; t\geq 0\} =\{N^{H}_{\beta_{t}}(a,b), \; t\geq 0\},	 \end{eqnarray*} where the parent process $N^{H}(a,b)$ is a fmfBm with parameters $a, b,$ $H\in (0, \, 1)$ and  the subordinator $\beta=\{ \beta_{t},\, t\geq 0\}$ is assumed to be independent of both the  Brownian motion and the fractional Brownian motion. If $H_1=\frac{1}{2} $ and $H_2= 0$, the process $L^{H_1H_2}_{\beta}(0,1)$ is called subordinated Brownian motion, it was investigated in \cite{Mag, Nan}. Also, the process $L^{H_1H_2}_{\beta}(0,1)$ is called subordinated fractional Brownian motion  it was investigated in  \cite{KGW, KWPS}.

Time-changed process is constructed by taking superposition of tow independent stochastic systems. The evolution of time in external process is replaced by a non-decreasing stochastic process, called subordinator.  The resulting time-changed process very often retain important properties of the external process, however certain characteristics might change. This idea of subordination was introduced by Bochner \cite{Bochner} and  was  explored in many papers (e.g. \cite{Al, AlE,  HmMl, MeHmMl,  KGW, MejMl, MeMl, Omer}).

The time-changed processes have found many interesting applications, for example in finance  \cite{ GZH, HJY, Omer, FSH}, in statistical inference \cite{AlexYES} and in physics  \cite{GSPE}.

Our goal in this parer is to study the main properties of  the time-changed fractional mixed fractional Brownian motion by inverse $\alpha$-stable subordinator paying attention to the long range dependence property.

\section{Main results and proofs}
We begin by defining the  inverse $\alpha$-stable subordinator.

\begin{Definition} The  inverse $\alpha$-stable subordinator $T^{\alpha}=\{T^{\alpha}_{t},\; t \geq 0\}$ is defined in the following way
	\begin{eqnarray}
	T^{\alpha}_{t}=inf\{r>0, \; \eta^{\alpha}_{r}\geq t\},
	\end{eqnarray}	
where $\eta^{\alpha}=\{ \eta^{\alpha}_{r},\; r\geq 0\}$ is the $\alpha$-stable subordinator  \cite{KSAT, MSOK} with Laplace transform
	\begin{eqnarray*}
	E(e^{-u\eta^{\alpha}_{r}})=e^{{-ru^{\alpha}}},\quad \alpha \in(0, 1).
	\end{eqnarray*}	
 The  inverse $\alpha$-stable subordinator  is a non-decreasing L\'{e}vy process, starting from zero, has a stationary and independent increments with  $\alpha$-self similar. Specially, when $\alpha \uparrow 1,$ $T^{\alpha}_{t}$ reduces to the physical time $t.$
\end{Definition}

 Let $T^{\alpha}$ be an  inverse $\alpha$-stable subordinator with index $\alpha\in(0, 1)$. From \cite{MMag, MAGD}, we know that \begin{eqnarray*}E(T^{\alpha}_{t})=\frac{t^{\alpha}}{\Gamma(\alpha+1)}\quad and \quad E((T^{\alpha}_{t})^{n})=\frac{t^{n\alpha}n!}{\Gamma(n\alpha+1)}.\end{eqnarray*}

\begin{Lemma}\label{lma2} Let $T^{\alpha}$ be an  inverse $\alpha$-stable subordinator with index $\alpha\in(0, 1)$ and $B^{H}$ be a fBm.
Then, by $\alpha$-self-similar and non-decreasing sample path of $T^{\alpha}_{t},$ we have
\begin{eqnarray*}
E(B_{T^{\alpha}_{t}})^{2}=\frac{t^{\alpha}}{\Gamma(\alpha+1)}\quad and \quad E(B^{H}_{T^{\alpha}_{t}})^{2}=\left(\frac{t^{\alpha}}{\Gamma(\alpha+1)}\right)^{2H}.
\end{eqnarray*}
\end{Lemma}
	\begin{proof} See  \cite{HJY, MAGD}.
	\end{proof}
\begin{Definition}
	Let $ N^{H_1H_2 }(a, b)=\{N^{{H_1H_2 }}_{t}(a,b), \; t\geq 0\}$  be a fmfBm and let  $T^{\alpha}$ be an inverse $\alpha$-stable subordinator with index $\alpha\in(0, 1)$. The subordinated of $N^{H_1H_2 }(a, b)$ by  means of  $T^{\alpha}$ is the process  $L_{T^{\alpha}}^{{H_1H_2 }}(a, b)=\{	L^{H_1H_2 }_{T^{\alpha}_t}, \; t\geq 0\}$   defined by:
	\begin{eqnarray}
		L^{H_1H_2 }_{T^{\alpha}_t}=N^{{H_1H_2 }}_{T^{\alpha}_{t}}(a,b)=aB^{{H_1 }}_{T^{\alpha}_{t}}+bB^{H_2 }_{T^{\alpha}_{t}}, \quad (a,b) \in R\times R\backslash\{0\},
	\end{eqnarray}
	where the subordinator $T^{\alpha}_t$ is assumed to be independent of both the Bm and the fBm.
\end{Definition}
\begin{Remark} \label{rem1}
When  $\alpha \uparrow 1,$ the processes $B_{T^{\alpha}_{t}}$ and $B^H_{T^{\alpha}_{t}}$ degenerate to  $B_{t}$ and $B^H_{t}.$
\end{Remark}

	\begin{Notation}
Let $U$ and $V$ be two centered random variables defined on the same probability space.  Let
	\begin{eqnarray}\label{qq12}
Corr(U, V)=\frac{Cov(U,V)}{\sqrt{E(U^{2})E(V^{2})}},
\end{eqnarray}	
denote the correlation coefficient between $U$ and $V.$
	\end{Notation}

Now we discuss the long range dependent behavior of    $L_{T^{\alpha}}^{H_1H_2}(a, b)$.
\begin{Definition}
	A finite variance stationary process $\{X_t,\;t\geq 0\}$ is said to have long range dependence property \cite{Cont}, if $\sum_{k=0}^{\infty}\gamma_k=\infty$, where
	\begin{eqnarray*}
		\gamma_k=Cov(X_k,X_{k+1}).
	\end{eqnarray*}
\end{Definition}

	In the following definition we give the equivalent definition for a non-stationary process $\{X_t,\;t\geq 0\}$.
\begin{Definition}\label{d1}
	Let $s>0$ be fixed and $t>s$. Then process $\{X_t,\;t\geq 0\}$ is said to have long range dependence property property if
	\begin{eqnarray*}
		Corr(X_t,X_s)\sim c(s)t^{-d}, \ \ as \ \ t\rightarrow\infty,
	\end{eqnarray*}	
	where $c(s)$ is a constant depending on $s$ and $d\in(0,1)$.
\end{Definition}
The main result can be stated as follows.
\begin{Theorem} \label{th1} Let $ N^{H_{1}H_{2}}(a, b)=\{N^{H_{1}H_{2}}_{t}(a,b), \; t\geq 0\}$  be the mixed fractional Brownian motion of parameters $a,  b$ and $H.$ Let   $T^{\alpha}=\{T^{\alpha}_{t},\; t\geq 0\}$ be an  inverse $\alpha$-stable subordinator with index $\alpha\in(0, 1)$ assumed to be  independent of both the Bm and the  fBm.
Then  the time-changed mixed fractional Brownian motion  by means of   $T^{\alpha}$   has long range dependence  property for every $H_{1}<H_{2}$ and $0<2\alpha H_{1}-\alpha H_{2}<1$.
	\end{Theorem}
	
	\begin{proof}
 Let   $T^{\alpha}=\{T^{\alpha}_{t},\; t\geq 0\}$ be an  inverse $\alpha$-stable subordinator with index $\alpha\in(0, 1)$ assumed to be  independent of both the Bm and the  fBm.  Let  $L_{T^{\alpha}}^{{H_{1}H_{2}}}(a, b)$ be the time-changed mixed fractional Brownian motion  by  means of the inverse $\alpha$-stable subordinator  $T^{\alpha}$  with index $\alpha\in(0, 1).$
The process $L_{T^{\alpha}}^{{H_{1}H_{2}}}(a, b)$ is not stationary hence  Definition \ref{d1}  will be used to establish the long range dependence property.\\

\textbf{Step 1:}  Let $s\leq t$. Since $B^{H_1}$ and $B^{H_2}$ has stationary increments, then we have
		\begin{eqnarray*}
Cov(	L^{H_{1}H_{2}}_{T^{\alpha}_t},	L^{H_{1}H_{2}}_{T^{\alpha}_s})
			&=&E(	L^{H_{1}H_{2}}_{T^{\alpha}_t}	L^{H_{1}H_{2}}_{T^{\alpha}_s})\nonumber
			= \frac{1}{2}E\left[(	L^{H_{1}H_{2}}_{T^{\alpha}_t})^2+(	L^{H_{1}H_{2}}_{T^{\alpha}_s})^2-(	L^{H_{1}H_{2}}_{T^{\alpha}_t}-L^{H_{1}H_{2}}_{T^{\alpha}_s})^2 \right]\nonumber\\
				 &=&\frac{1}{2}E\left[(N^{{H_{1}H_{2}}}_{T^{\alpha}_{t}}(a,b))^2+(N^{{H_{1}H_{2}}}_{T^{\alpha}_{s}}(a,b))^2-(N^{{H_{1}H_{2}}}_{T^{\alpha}_{t}}(a,b)-N^{{H_{1}H_{2}}}_{T^{\alpha}_{s}}(a,b))^2\right ]	 \nonumber\\
			&=&\frac{1}{2} E\left[ (aB^{H_{1}}_{T^{\alpha}_{t}}+bB^{H_{2}}_{T^{\alpha}_{t}}) ^2
			+( aB^{H_{1}}_{T^{\alpha}_{s}}+bB^{H_{2}}_{T^{\alpha}_{s}}) ^2\right]
		\nonumber\\
			&&-\frac{1}{2} E\left[\left( a(B^{H_{1}}_{T^{\alpha}_{t}}-B^{H_{1}}_{T^{\alpha}_{s}})+b(B^{H_{2}}_{T^{\alpha}_{t}}-B^{H_{2}}_{T^{\alpha}_{s}})\right)^2\right]\nonumber\\
			&=&\frac{1}{2}E \left[ (aB^{H_{1}}_{T^{\alpha}_{t}}+bB^{H_{2}}_{T^{\alpha}_{t}}) ^2
			+( aB^{H_{1}}_{T^{\alpha}_{s}}+bB^{H_{2}}_{T^{\alpha}_{s}}) ^2\right]\\&-&\frac{1}{2}E \left[( aB^{H_{1}}_{T^{\alpha}_{t-s}}+bB^{H_{2}}_{T^{\alpha}_{t-s}})^2\right]
			\nonumber\\
			&=&\frac{1}{2}E\left[(aB^{H_1}_{T^{\alpha}_{t}}+( bB^{H_2}_{T^{\alpha}_{t}})^2+2(aB^{H_1}_{T^{\alpha}_{t}} bB^{H_2}_{T^{\alpha}_{t}}) \right]\nonumber\\
			&&+\frac{1}{2}E\left[(aB^{H_1}_{T^{\alpha}_{s}})^2+( bB^{H_2}_{T^{\alpha}_{s}})^2+2(aB^{H_1}_{T^{\alpha}_{s}} bB^{H_2}_{T^{\alpha}_{s}}) \right]\nonumber\\
			&&-\frac{1}{2}E\left[(aB^{H_1}_{S^{\lambda,\alpha}_{t-s}})^2+( bB^{H_2}_{S^{\lambda,\alpha}_{t-s}})^2
			+2(aB^{H_1}_{T^{\alpha}_{t-s}} bB^{H_2}_{T^{\alpha}_{t-s}}) \right].\nonumber
	\end{eqnarray*}	
Since $B_{t}^{H_1}$ and $B_{t}^{H_2}$ are independent and using Lemma \ref{lma2} we get
	\begin{eqnarray*}	
E(	L^{H_1H_2}_{T^{\alpha}_t} L^{H_1H_2}_{T^{\alpha}_s})
			&=&\frac{a^2}{2}\left[E(B^{H_1}_{T^{\alpha}_{t}})^2+E(B^{H_1}_{T^{\alpha}_{s}})^2-E(B^{H_1}_{T^{\alpha}_{t-s}})^2 \right]\nonumber\\
			&&+\frac{b^2}{2}\left[E(B^{H_2}_{T^{\alpha}_{t}})^2+E(B^{H_2}_{T^{\alpha}_{s}})^2-E(B^{H_2}_{T^{\alpha}_{t-s}})^2 \right]\nonumber\\
			 &=&\frac{a^2}{2}\left[\left(\frac{t^{\alpha}}{\Gamma(\alpha+1)}\right)^{2{H_1}}+\left(\frac{s^{\alpha}}{\Gamma(\alpha+1)}\right)^{2{H_1}}-\left(\frac{(t-s)^{\alpha}}{\Gamma(\alpha+1)} \right)^{2{H_1}}\right]\nonumber\\
			 &&+\frac{b^2}{2}\left[\left(\frac{t^{\alpha}}{\Gamma(\alpha+1)}\right)^{2{H_2}}+\left(\frac{s^{\alpha}}{\Gamma(\alpha+1)}\right)^{2{H_2}}-\left(\frac{(t-s)^{\alpha}}{\Gamma(\alpha+1)} \right)^{2{H_2}}\right]\nonumber\\
			&=&\frac{a^2\left[t^{2\alpha {H_1}}+s^{2\alpha {H_1}}-(t-s)^{2\alpha {H_1}} \right]}{2[\Gamma(\alpha+1)]^{2{H_1}}}\nonumber
			+\frac{b^2\left[t^{2\alpha {H_2}}+s^{2\alpha {H_2}}-(t-s)^{2\alpha {H_2}} \right]}{2[\Gamma(\alpha+1)]^{2{H_2}}}.\nonumber
				\end{eqnarray*}	
Hence for all $s\leq t$ and $H_1<H_2$ we have
	\begin{eqnarray}
	E(	L^{H_1H_2}_{T^{\alpha}_t}	L^{H_1H_2}_{T^{\alpha}_s})=\frac{a^2\left[t^{2\alpha {H_1}}+s^{2\alpha {H_1}}-(t-s)^{2\alpha {H_1}} \right]}{2[\Gamma(\alpha+1)]^{2{H_1}}}
			+\frac{b^2\left[t^{2\alpha {H_2}}+s^{2\alpha {H_2}}-(t-s)^{2\alpha {H_2}} \right]}{2[\Gamma(\alpha+1)]^{2{H_2}}}.\label{q1}
		\end{eqnarray}
 \textbf{Step 2:} Let  $s$ be  fixed. Then by  Taylor's expansion we have for large $t$
\begin{eqnarray*}
	E(	L^{H_1H_2}_{T^{\alpha}_t} L^{H_1H_2}_{T^{\alpha}_s})&\sim& \frac{a^2}{2[\Gamma(\alpha+1)]^{2{H_1}}}t^{2\alpha {H_1}}\left[2\alpha {H_1} \frac{s}{t}+s^{2\alpha {H_1}}t^{-2\alpha {H_1}}+O(t^{-2}) \right]
			\\&&+\frac{b^2}{2[\Gamma(\alpha+1)]^{2{H_2}}}t^{2\alpha {H_2}}\left[2\alpha {H_2} \frac{s}{t}+s^{2\alpha {H_2}}t^{-2\alpha {H_2}}+O(t^{-2}) \right]\\&\sim& \frac{a^2t^{2\alpha {H_1}}}{2[\Gamma(\alpha+1)]^{2{H_1}}}\left[2\alpha {H_1} \frac{s}{t}+(\frac{s}{t})^{2\alpha {H_1}}+O(t^{-2}) \right]
			\\&&+\frac{b^2t^{2\alpha {H_2}}}{2[\Gamma(\alpha+1)]^{2{H_2}}}\left[2\alpha {H_2} \frac{s}{t}+(\frac{s}{t})^{2\alpha {H_2}}+O(t^{-2}) \right]
 \\&\sim& \frac{a^{2}\alpha s}{(\Gamma(\alpha+1))^{2{H_1}}}t^{2\alpha H_1-1}+\frac{b^{2}\alpha s}{(\Gamma(\alpha+1))^{2{H_2}}}t^{2\alpha H_2-1}.
		\end{eqnarray*}
Then  for fixed $s$ and large $t$, $	L^{H_1H_2}_{T^{\alpha}_t}$ satisfies
	\begin{eqnarray}
	E(	L^{H_1H_2}_{T^{\alpha}_t} L^{H_1H_2}_{T^{\alpha}_s})\sim \frac{a^{2}\alpha s}{(\Gamma(\alpha+1))^{2{H_1}}}t^{2\alpha H_1-1}+\frac{b^{2}\alpha s}{(\Gamma(\alpha+1))^{2{H_2}}}t^{2\alpha H_2-1}. \label{q33}
		\end{eqnarray}
 \textbf{Step 3:}  Let  $H_1<H_2$. Using Eqs. \eqref{qq12}, \eqref{q33} and by  Taylor's expansion we get, as $t\rightarrow\infty$
	\begin{eqnarray*}	
Corr(	L^{H_1H_2}_{T^{\alpha}_t},	L^{H_1H_2}_{T^{\alpha}_s})&\sim&  
\frac{\frac{a^{2}\alpha s}{(\Gamma(\alpha+1))^{2{H_1}}}t^{2\alpha H_1-1}+
\frac{b^{2}\alpha s}{(\Gamma(\alpha+1))^{2{H_2}}}t^{2\alpha H_2-1}}{
\left[\frac{a^{2}\alpha }{(\Gamma(\alpha+1))^{2H_1}}t^{2\alpha H_1}+\frac{b^{2}\alpha }{(\Gamma(\alpha+1))^{2H_2}}t^{2\alpha H_2}\right]^{\frac{1}{2}} \left[E(L_s^{T^{\alpha}})^{2}\right]^{\frac{1}{2}}}    \\&=&
\frac{\frac{a^{2}\alpha s}{(\Gamma(\alpha+1))^{2H_1}}t^{2\alpha H_1-1}+\frac{b^{2}\alpha s}{(\Gamma(\alpha+1))^{2H_2}}t^{2\alpha H_2-1}}{\frac{|b|\alpha^{\frac{1}{2}} t^{\alpha H_2} }{(\Gamma(\alpha+1))^{H_2}}
\left[\frac{a^{2}}{2b^{2}(\Gamma(\alpha+1))^{1-2H_2}}t^{ 2\alpha H_1-2\alpha H_2}+1\right]^{\frac{1}{2}}\left[ E(L_s^{T^{\alpha}})^{2}\right]^{\frac{1}{2}}}   \\&\sim&  \frac{a^{2}\alpha^{\frac{1}{2}} st^{2\alpha H_1- \alpha H_2-1}}{|b|(\Gamma(\alpha+1))^{2H_1-H_2}
\left[ E(L_s^{T^{\alpha}})^{2}\right]^{\frac{1}{2}}} + \frac{|b|\alpha^{\frac{1}{2}} st^{\alpha H_2-1}}{
(\Gamma(\alpha+1))^{H_2}\left[ E(L_s^{T^{\alpha}})^{2}\right]^{\frac{1}{2}}} .		
	\end{eqnarray*}
Hence, for every $H_1<H_2$ we have
\begin{eqnarray}	\label{qq13}
Corr(	L^{H_1H_2}_{T^{\alpha}_t},	L^{H_1H_2}_{T^{\alpha}_s})\sim  \frac{a^{2}\alpha^{\frac{1}{2}} s\,t^{2\alpha H_1- \alpha H_2-1}}{|b|(\Gamma(\alpha+1))^{2H_1-H_2}
\left[ E(L_s^{T^{\alpha}})^{2}\right]^{\frac{1}{2}}} + \frac{|b|\alpha^{\frac{1}{2}} s\,t^{\alpha H_2-1}}{
(\Gamma(\alpha+1))^{H_2}\left[ E(L_s^{T^{\alpha}})^{2}\right]^{\frac{1}{2}}}. 	
\end{eqnarray}	
Then the correlation function of $L^{H_1H_2}_{T^{\alpha}_t}$ decays like a mixture of power law $t^{-(2\alpha H_1- \alpha H_2-1)}+t^{-(1-\alpha H_2)}$. Since $0<2\alpha H_1- \alpha H_2<1$ then the first term tends to zero as $t\rightarrow\infty.$  Then the time-changed process $L_{T^{\alpha}}^{H}(a, b)$ exhibits  long range dependence  property for all $H_1<H_2$ and $0<2\alpha H_1- \alpha H_2<1$.
\end{proof}	
	
\begin{Remark} When $a=0$ and $b=1$ in Eqs. \eqref{q33} and \eqref{qq13} we get
		\begin{eqnarray*}
				&&E(	L^{H_1H_2}_{T^{\alpha}_t}	L^{H_1H_2}_{T^{\alpha}_s})=E(B^{H_2}_{T_{t}^{\alpha}}B^{H_2}_{T_{s}^{\alpha}})
				\sim \frac{\alpha st^{2\alpha H_2-1}}{(\Gamma(\alpha+1))^{2H_2}}, \ \ as \ \ t \rightarrow \infty,
				\\&& Corr(	L^{H_1H_2}_{T^{\alpha}_t},	L^{H_1H_2}_{T^{\alpha}_s})=Corr(B^{H_2}_{T_{t}^{\alpha}},B^{H_2}_{T_{s}^{\alpha}})
				\sim \frac{\alpha^{\frac{1}{2}}st^{\alpha H_2-1}}{(\Gamma(\alpha+1))^{H_2}\sqrt{E(B^{H_2}_{T_{s}^{\alpha}})^{2}}}, \ \ as \ t\rightarrow\infty.
			\end{eqnarray*}
\end{Remark}

Hence we obtain the following result.
\begin{Corollary}\label{cor1}
The fractional Brownian motion time changed by  inverse $\alpha$-stable subordinator with index $\alpha\in(0, 1)$ is of long range dependence   for the Hurst exponent  $H\in(0,1)$.
\end{Corollary}

Similar result as Corollary  \ref{cor1} was obtained in \cite{KGW} (\cite{KWPS}) in the case of fractional Brownian motion time changed by tempered stable subordinator (gamma subordinator).\\

As application to the original process  we obtain the following. .

\begin{Corollary}\label{cor2}
Let $H>\frac{1}{2}.$ When  $\alpha \uparrow 1,$  in Eqs. \eqref{q33} and \eqref{qq13}  we have
		\begin{eqnarray*}
		&&\lim_{\alpha \rightarrow 1}E(	L^{H_1H_2}_{T^{\alpha}_t}	L^{H_1H_2}_{T^{\alpha}_s})=\frac{a^{2}s}{2}+b^{2}st^{2H-1},    \ \ as \ \  t\rightarrow\infty,
\\&&\lim_{\alpha \rightarrow 1}	Corr(	L^H_{T^{\alpha}_t},	L^{H_1H_2}_{T^{\alpha}_s})=\frac{a^{2} s\,t^{- H}}{2|b|
\sqrt{ E(N_s^H(a,b))^{2}}} + \frac{|b| s\,t^{ H-1}}{
\sqrt{ E(N_s^H(a,b))^{2}}},\ \ as \ \ t\rightarrow\infty.
	\end{eqnarray*}	
\end{Corollary}

Hence using Remark \ref{rem1} and corollary \ref{cor2} we can see that the mixed fractional Brownian motion of parameters $a,  b $ and $  H$ has long range dependence  property  for all $H>\frac{1}{2}$ in  sense of Definition \ref{d1}.

\begin{Remark}
\begin{enumerate}
\item Let $H\in(0, 1)$. Then  \begin{eqnarray}	\label{qq111}
Corr(B^{H}_{t}, B^{H}_{s})\sim   \frac{ st^{ H-1}}{
\sqrt{E(B^{H}_{s})^{2}} },\ \ as \ \ t\rightarrow\infty.   	
	\end{eqnarray}
Indeed, we take  $a=0 $ and $b=1$ in Eq. \eqref{qq13}. When  $\alpha \uparrow 1$ and using Remark \ref{rem1} we obtain Eq. \eqref{qq111}.

\item When  $\alpha \uparrow 1,$  in Eq. \eqref{q1} we have
		\begin{eqnarray*}
		\lim_{\alpha \rightarrow 1}	E(	L^{H_1H_2}_{T^{\alpha}_t}	L^{H_1H_2}_{T^{\alpha}_s})=\frac{a^2}{2}\left[t^{2H_1}+s^{2 H_1}-(t-s)^{2 H_1} \right]
			+\frac{b^2}{2}\left[t^{2H_2}+s^{2 H_2}-(t-s)^{2 H_2} \right].
	\end{eqnarray*}	
 \end{enumerate}
\end{Remark}
\begin{Corollary}
The fractional mixed fractional Brownian motion  has long range dependence for  every $H_1 < H_2.$ 
\end{Corollary}
The  idea, used  results for the time-changed process to obtain  a results for the original one is already investigated in \cite{HmMl}.\\

The fmfBm has been further generalized by Th$\ddot{a}$le in 2009 \cite{{Thale}} to the generalized mixed fractional Brownian motion. A generalized mixed fractional Brownian motion of parameter $H=(H_1, H_2, ..., H_n)$ and $\alpha=(\alpha_1, \alpha_2, ..., \alpha_n)$ is a stochastic process $Z=(Z_t^{H, \alpha})_{t\geq 0}$ defined by
 \begin{eqnarray*}
Z_t^{H, \alpha}=\alpha_1 B_{t}^{H_{1}}+\alpha_2B_{t}^{H_{2}}+...+\alpha_nB_{t}^{H_{n}}
 \end{eqnarray*}
Forthcoming work, we will investigate the long range dependence property of the time-changed generalized mixed fractional Brownian motion by inverse $\alpha$-stable subordinator \cite{MLKE}.

\end{document}